\documentclass[a4paper]{amsart}

\usepackage{amscd,graphicx,epsfig}
\usepackage[usenames,dvipsnames]{xcolor} 
\usepackage[colorlinks=true, pdfstartview=FitV, linkcolor=blue, citecolor=blue, urlcolor=blue]{hyperref}
\usepackage{dsfont} 
\usepackage{textcomp}
\usepackage{amssymb}
\usepackage{amsmath}
\usepackage{amsfonts}
\usepackage{mathrsfs}
\usepackage{enumerate}
\usepackage[T1]{fontenc}
\usepackage[latin1]{inputenc}
\usepackage{nicefrac}

\usepackage{tikz,pgffor}
\usetikzlibrary{arrows}
\usepackage[all]{xy}
\usepackage{hyperref}
\usetikzlibrary{patterns}

\theoremstyle{plain}
\newtheorem{theorem}{Theorem}[section]

\newtheorem{lemma}[theorem]{Lemma}
\newtheorem{proposition}[theorem]{Proposition}

\newtheorem*{theorem*}{Theorem}

\theoremstyle{definition}

\newtheorem{example}[theorem]{Example}

\theoremstyle{remark}

\newtheorem{remark}[theorem]{Remark}

\renewcommand{\O}{\mathcal{O}}
\def\Z{{\mathds Z}}

\def\Q{{\mathds Q}}

\def\V{{\mathcal V}}

\def\QQ{{\mathcal Q}}

\def\D{{\mathfrak{D}}}
\def\E{{\mathscr{E}}}
\def\T{{\mathbb{T}}}

\def\ord{{\rm ord}}
\def\P{{\mathds{P}}}

\def\F{{\mathcal{F}}}

\def\Ray{{\mathrm{Ray}}}
\def\Vert{{\mathrm{Vert}}}

\def\Cl{{\mathrm{Cl}}}
\def\div{{\mathrm{div}}}

\def\XX{{\mathscr{X}}}

\def\SS{{\mathcal{S}}}

\address{Kevin Langlois, Mathematisches Institut, Heinrich Heine Universit\"{a}t, 40225 D\"{u}sseldorf, Germany}
\email{langlois.kevin18@gmail.com}

\address{Ronan Terpereau, Max Planck Institut f\"{u}r Mathematik, Vivatsgasse 7, 53111 Bonn, Germany}
\email{rterpere@mpim-bonn.mpg.de}


\begin{document}

\title[The Cox ring 
of a complexity-one horospherical variety]{The Cox ring of a complexity-one \\ horospherical variety}
\author{Kevin Langlois}
\author{Ronan Terpereau}

\begin{abstract}  
Cox rings are intrinsic objects naturally generalizing homogeneous coordinate rings of projective spaces. A complexity-one horospherical variety is 
a normal variety equipped with a reductive group action whose general orbit is horospherical and of codimension one.
In this note, we provide a presentation by generators and relations for the Cox rings of complete rational complexity-one
horospherical varieties.   
\end{abstract}  

\maketitle

\setcounter{tocdepth}{1}
\tableofcontents

\section*{Introduction}

All algebraic varieties and algebraic groups considered in this article are defined over an algebraically closed field $k$ of characteristic zero.\\  

Let $G$ be a connected simply-connected reductive algebraic group (i.e., a direct product of a torus and a connected simply-connected semisimple group), and let $H \subseteq G$ be a closed subgroup. 
The homogeneous space $G/H$ is called \emph{horospherical} if $H$ contains a maximal unipotent subgroup of $G$. 
Geometrically, the homogeneous space $G/H$ may be  realized as the total space of a principal $\T$-bundle over the flag variety $G/P$, where $P=N_G(H)$  is the parabolic subgroup normalizing $H$ and $\T$ is the algebraic torus $P/H$. 

In this paper, we consider a specific class of $G$-varieties: the \emph{complexity-one horospherical $G$-varieties}, that is, the normal $G$-varieties whose general orbit is horospherical and of codimension one. 
From the Luna--Vust theory \cite{LV83}, there is a combinatorial description of such varieties (see \cite{Tim97}, \cite[Ch. 16]{Tim11}, and \cite[\S 1]{LT16}) which is quite similar to the classical one of torus embeddings (see for instance \cite{Ful93}). The geometry of complexity-one horospherical varieties has been studied in \cite{LT16}; see \cite[Th. 2.5 and 2.6]{LT16} for a smoothness criterion,  \cite[Cor. 2.12]{LT16} for a description of the class group, and \cite[Th. 2.18]{LT16} for a description of the canonical class. Also the stringy invariants of these varieties have been determined in \cite{LPR}.

An important issue for the theory of complexity-one horospherical varieties is to describe them in terms of equations via \lq explicit coordinates\rq. In the special framework of torus actions of complexity one, this program was achieved in some cases via the theory of Cox rings (see \cite{HS10, HH13}).

Let $X$ be a normal variety whose class group $\Cl(X)$ is of finite type and such that $\Gamma(X,{\O}_{X}^{\times})= k^\times$. As a graded $k$-vector space, the \emph{Cox ring} of $X$ is defined as
$$R(X) = \bigoplus_{[D]\in \Cl(X)}\Gamma(X, \O_{X}(D)).$$
The vector space $R(X)$ can be equipped with a multiplicative law making $R(X)$ a $\Cl(X)$-graded algebra over $k$; see \cite[\S 1.4]{ADHL15} for details. Let us note that any projective $\Q$-factorial normal variety $X$, with finitely generated class group $\Cl(X)$, is completely determined (up to isomorphism) by the data of its Cox ring, as a $\Cl(X)$-graded algebra, and an ample class (see \cite[\S 1.6.3]{ADHL15}).

There are deep connections between Cox rings, invariant theory, and the minimal model program (see \cite{McK10} for a survey). Cox rings also appear in the classification of Fano varieties (see for instance \cite{Her14,Sue14} for complexity-one torus actions), and they have applications in arithmetic geometry to the study of rational points of algebraic varieties (see \cite[\S 6]{ADHL15} for an overview).

The Cox ring has been computed for several important classes 
of algebraic varieties with reductive group action; see \cite{KR87} for flag varieties, \cite{Cox95} for toric varieties, and more generally \cite{Bri07} and \cite{Gag14} for spherical varieties (complexity zero case).
A description of the Cox ring for algebraic varieties with torus action is given in \cite{HS10}.

The purpose of this article is to describe the Cox ring of a new class of algebraic varieties with reductive group action, namely our main result (Theorem \ref{cox}) is a description of the Cox ring of any complete rational complexity-one horospherical variety by generators and relations. 
Note that, by \cite[Cor. 2.12]{LT16}, the class group of a complexity-one horospherical variety is finitely generated if and only if the variety is rational. The completeness assumption however is only by convenience. 

The proof of Theorem \ref{cox} is based on the fact that a complexity-one horospherical variety $X$ is naturally equipped with an action of the algebraic torus $\T$; see Lemma \ref{Taction}. We first describe the $\T$-action on $X$ in terms of divisorial fans by adapting the results of Altmann--Kiritchenko--Petersen in \cite{AKP15} obtained for spherical varieties of minimal rank; see Proposition \ref{Tvariety}. Then our result follows from \cite[Th. 4.8]{HS10} which describes Cox rings of $\T$-varieties in terms of their divisorial fans.

\section{Brief overview of the combinatorics}
Let us introduce the necessary background on Luna--Vust theory required to express and prove Theorem \ref{cox}. Here we give some geometric ideas how it works out; we refer to \cite[Ch. 16]{Tim11} and \cite[\S 1]{LT16} for precise statements.
The equivariant birational type of a rational complexity-one horospherical $G$-variety $X$ has a simple description. Indeed, by \cite[Satz 2.2]{Kno90}, there exists a $G$-equivariant birational map 
\begin{equation}\label{eqr}
\psi: X \dashrightarrow \P^1 \times G/H,
\end{equation}
where $G$ acts by left multiplication on the horospherical homogeneous space $G/H$ and trivially on the projective line $\P^1$.

The combinatorics introduced thereafter are classifying objects for a specific category: the \emph{category of $G$-models} of $\P^1 \times G/H$, whose objects are pairs $(X, \psi)$, where $X$
is a normal $G$-variety and $\psi$ is a $G$-equivariant birational map as in \eqref{eqr}.
Morphisms $(X_{1},\psi_{1})\rightarrow (X_{2}, \psi_{2})$ in this category are $G$-morphisms $f:X_{1}\rightarrow X_{2}$ such that $\psi_2\circ f = \psi_1$. In the following we will omit the \emph{base rational map} $\psi$ and write $X$ for $(X,\psi)$. 

Let $X$ be  a complete $G$-model of $\P^1 \times G/H$.
Fix a Borel subgroup $B \subseteq G$ whose unipotent radical is contained in $H$. The $B$-stable prime divisors of $X$ which are not $G$-stable are called \emph{colors} of $X$; the finite set of colors of $X$ is denoted by $\F_X$. This set is 
in one-to-one correspondence with the set of Schubert divisors of $G/P$,
where $P$ is the normalizer of $H$ in $G$. 
By \cite[Prop. 2.9]{LT16}, there exists a unique proper morphism of $G$-models
$$\pi: \tilde{X} \to X,$$ called \emph{discoloration} of $X$, such that: the colors of $\tilde{X}$ do not contain a $G$-orbit; and for any $G$-stable irreducible closed subvariety $Z\subseteq \tilde{X}$, the sets of essential $G$-invariant 
valuations of the field of rational functions $k(X)$ describing the local rings of $Z$ and $\pi(Z)$ are the same (see \cite[\S 8]{LV83}).  
The morphism $\pi$ is a resolution
of the indeterminacy locus of the $G$-equivariant dominant rational map $\psi':X \dashrightarrow  G/P$
induced by $\psi$, that is, $\psi' \circ \pi$ is a morphism.
The preimage of $P/P$ by $\psi' \circ \pi$ defines a $P$-variety $Y \subseteq \tilde{X}$ which
turns out to be a normal $\T$-variety with general orbit of codimension one and $\T$-isomorphic to $\T$. Moreover, the natural morphism
$$G\times Y\rightarrow \tilde{X},\,\,(g,x)\mapsto g\cdot x$$
induces a $G$-isomorphism between the parabolic induction $G\times^{P}Y$ of $Y$ and the $G$-variety $\tilde{X}$.
Thus, $X$ is determined by the $\T$-variety $Y$ and by some subsets of $\F_X$ corresponding to the $G$-stable subvarieties that $\pi$ contracts.  \\

Since the combinatorial datum of $Y$ will be part of the one of $X$, we make a short digression to explain the combinatorial description of the $\T$-varieties following Altmann--Hausen--S{\"u}ss; see \cite{AHS08} for details.  
Let $\XX$ be a normal affine $\T$-variety (with $\T$ acting faithfully on $\XX$), let $M$ be the group of characters of $\T$, and let $M_\Q=\Q\otimes_\Z M $. The coordinate ring $k[\XX]=\Gamma(\XX,\O_\XX)$ is described by a piecewise linear map, called \emph{polyhedral divisor} \cite[Def. 2.7]{AH06}, and defined as 
\begin{equation}
\D:   \omega  \to  {\rm CaDiv}_\Q(V), \; \; \;
 m  \mapsto  \hspace{-5mm} \sum_{\substack{Z \subseteq V \\ Z  \text{ prime divisor}}}  \hspace{-5mm} \min \langle \D_Z, m \rangle \cdot Z,
\end{equation}
where $\omega \subseteq M_\Q$ is a polyhedral cone spanning $M_\Q$, each $\D_Z$ is a prescribed polyhedron lying in the dual of $M_\Q$, 
and ${\rm CaDiv}_\Q(V)= \Q \otimes_\Z {\rm CaDiv}(V)$ is the $\Q$-vector space generated by the group of Cartier divisors of a certain normal variety $V$ obtained as a rational quotient of $\XX$ by $\T$. The $\T$-action on $\XX$ induces an $M$-grading of algebra on $k[\XX]$ with weight cone $\omega$. Moreover, each graded piece of $k[\XX]$ corresponding to $m\in \omega\cap M$ identifies with $\Gamma(V, \mathcal{O}_{V}(\D(m)));$
see \cite[Th. 3.1 and 3.4]{AH06} for a precise statement.

As $\D$ is determined by the $\D_Z$'s, one usually denotes 
$$ \D= \sum_{Z \subseteq V} \D_Z \cdot Z.$$
For all but a finite numbers of prime divisors $Z \subseteq V$, the $\D_Z$'s are equal to the dual polyhedral cone $\sigma=\omega^\vee$; the latter is called the \emph{tail} of $\D$.
We follow the conventions of \cite[\S 2]{AHS08} and will say that $\D$ is defined over a compactification  $\bar{V}$ of $V$ by adding empty coefficients on the boundary if necessary.

In the general case where the normal $\T$-variety $\XX$ is not affine, there is an open covering of $\XX$ by affine $\T$-varieties \cite[Cor. 3.2]{Sum74}. Therefore, $\XX$ is described by a finite set $\SS$ of polyhedral divisors defined over a common complete variety $\bar{V}$ and satisfying certain compatibility conditions \cite[Th. 5.6]{AHS08}; the set $\SS$ is called a \emph{divisorial fan} \cite[Def. 5.1]{AHS08}. 
Let us note that, for any prime divisor $Z\subseteq \bar{V}$, the set of all coefficients $\D_{Z}$ when $\D$ runs through $\SS$ defines a polyhedral subdivision $\SS_{Z}$ of $M_\Q$; it is called a \emph{slice} of $\SS$ over $Z$. The \emph{support} of $\SS$ is the set of prime divisors of $\bar{V}$ where the slices
are non-empty and non trivial. 
A \emph{vertex} of a slice $\SS_{Z}$ is a vertex of one of its elements.\\

We now return to the combinatorial description of the complexity-one horospherical $G$-variety $X$. Let $\E'$ be a divisorial fan over $\P^{1}$ describing the $\T$-variety $Y$ defined earlier. Each $\D \in \E'$ defines a dense open subset in $Y$ and a $G$-stable dense open subset in $\tilde{X}$ via parabolic induction;
we denote the latter by $\tilde{X}(\D)$. 
Let $\F\subseteq \F_{X}$ be the set of colors of $X$ containing a $G$-orbit which is the image by $\pi$ of a $G$-orbit in the exceptional locus of $\pi_{|\tilde{X}(\D)}$. 
The pair $(\D,\F)$ is called a \emph{colored polyhedral divisor} and describes
the $G$-stable dense open subset $X(\D,\F):=\pi(\tilde{X}(\D))$.

The finite set $\E$ of colored polyhedral divisors $(\D,\F)$ obtained from $\E'$ and $\pi$ as above 
is called a \emph{colored divisorial fan} associated with $X$.
This set constitutes the combinatorial counterpart of $X$ as explained in \cite[\S 1]{LT16}. Many geometric properties of $X$ are reflected in the combinatorial object $\E$; we refer for instance to \cite[Th. 1.10, 2.5, and 2.6]{LT16} for characterizations of completeness and smoothness properties.

In \cite[\S 2.3]{LT16} a description of the 
prime $G$-divisors of $X$ in terms of $\E$ is given. They are separated into two sets depending on whether they have an open $G$-orbit or not. The prime $G$-divisors having an open $G$-orbit are in bijection with the set $\Vert(\E)$, and the other prime $G$-divisors are in bijection with the set $\Ray(\E)$. 
The elements of $\Vert(\E)$ are pairs $(y,v)$, where $y\in \P^{1}$ and $v$ is a vertex of the slice $\E_{y}$.  
The elements of $\Ray(\E)$ are certain extremal rays of tails of elements of $\E$ satisfying additional combinatorial conditions; see \cite[\S 2.3.1]{LT16} for details.

\section{Statement of the main result}
Let $X$ be a complete rational complexity-one horospherical $G$-variety with general orbit $G$-isomorphic to the homogeneous space $G/H$ and
colored divisorial fan $\E$. The main result of this paper is the following.

\begin{theorem} \label{cox}
Let $\{y_{1}, \ldots, y_{r}\}\subseteq \P^{1}$ be the support of the colored divisorial fan $\E$. The Cox ring of the complexity-one horospherical variety $X$
is isomorphic to
$$ R(G/P) \otimes_k  k \left[S_{\rho} ;\, \rho\in\Ray(\E) \right] \otimes_{k}  k\left[T_0,T_1,T_{(y_{i}, v)} \,;\, (y_{i}, v)\in \Vert(\E)\right]/I,$$
where $I$ is the ideal generated by the elements 
$$-\alpha_i T_0-\beta_i T_1 + \prod_{v \text{ vertex of } \E_{y_i}}  T_{(y_{i},v)}^{\mu(v)}$$ 
for $1\leq i\leq r$, with $y_i=[\alpha_i:\beta_i]$, and $\mu(v)$ is the smallest integer $d \in \Z_{>0}$ such that $dv$ is a lattice vector.
Moreover, the $\Cl(X)$-degree of the variables $S_\rho$ and $T_{(y_i,v)}$ is given by the class of the prime $G$-divisors corresponding to $\rho$ and $(y_i,v)$ respectively, and the $\Cl(X)$-grading on $R(G/P)$ 
is obtained by identifying colors of $X$ and Schubert divisors of $G/P$.
\end{theorem}

\begin{remark}\footnote{We are grateful to the referee for suggesting this remark.}
In the case where $r\geq 2$ the variables $T_{0}, T_{1}$ can be eliminated and the presentation of the Cox ring $R(X)$ in Theorem \ref{cox} takes the following form. Denote by $\Phi$ a basis of the $(r-2)$-dimensional vector space 
$$\left\{(\lambda_{1},\ldots, \lambda_{r})\in k^{r}\,\vrule \,\sum_{i=1}^{r}\lambda_{i}(\alpha_{i}, \beta_{i})= 0\right\}$$ and by $F_{i}$ the monomial $\prod_{v}  T_{(y_{i},v)}^{\mu(v)}$. Then $R(X)$ is isomorphic to
$$ R(G/P) \otimes_k  k \left[S_{\rho} ;\, \rho\in\Ray(\E) \right] \otimes_{k}  k\left[T_{(y_{i}, v)} \,;\, (y_{i}, v)\in \Vert(\E)\right]/J,$$
where $J$ is the ideal generated by the elements $\sum_{i=1}^{r}\gamma_{i}F_{i}$ for $(\gamma_{1},\ldots, \gamma_{r})\in \Phi$. 
\end{remark}

The reader is referred to \cite[\S 3.2.3]{ADHL15} for a presentation by generators and relations of the Cox ring of a flag variety. Note that our result implies that the Cox ring of a complete rational complexity-one horospherical variety is finitely generated.

\begin{example}
Let 
$G={\rm SL}_3$ and let $H$ be a maximal unipotent subgroup of $G$. The parabolic subgroup $P=N_G(H)$ is a Borel subgroup of $G$ and $\T$ is a $2$-dimensional torus; in particular $M \cong \Z^2$.
 \\
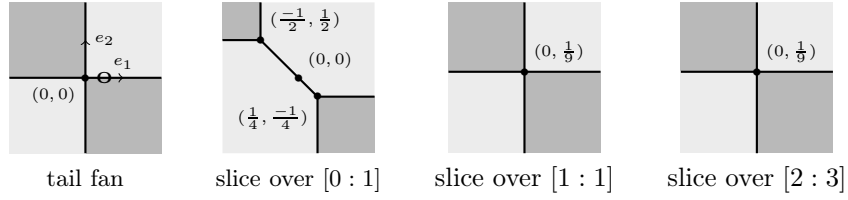
\begin{center}
 \begin{tabular}{ccccccc}
\begin{tikzpicture} 
\fill[gray!55] [-, thin] (0,0) -- (0,1) -- (-1,1) -- (-1,0) -- cycle;
\fill[gray!15] [-, thin] (0,0) -- (0,1) -- (1,1) -- (1,0)  -- cycle;
\fill[gray!55] [-, thin] (0,0) -- (1,0) -- (1,-1) -- (0,-1) -- cycle;
\fill[gray!15] [-, thin] (0,0) -- (0,-1) -- (-1,-1) -- (-1,0) -- cycle;
\draw (0,0) node[below left]{{\tiny $(0,0)$}} node{{\tiny $\bullet$}};
\draw (0.8,0.8)  node{};
\draw (0.8,-0.8)  node{};
\draw (-0.8,-0.8)  node{};
\draw (-0.8,0.8)  node{};
\draw [-, thick] (0,0) -- (0,1);
\draw [-, thick] (0,0) -- (-1,0);
\draw [-, thick] (0,0) -- (0,-1);
\draw [-, thick] (0,0) -- (1,0);
\draw [->] (0,0) -- (0,0.5);
\draw [->] (0,0) -- (0.5,0);
\draw (0.5,0)  node[above]{{\tiny $e_1$}};
\draw (0,0.5)  node[right]{{\tiny $e_2$}};
\draw (0.25,0) node{{\tiny \textbf{O}}};
\end{tikzpicture}& & 
\begin{tikzpicture} 
\fill[gray!55] [-, thin] (-0.5,0.5) -- (-0.5,1) -- (-1,1) -- (-1,0.5) -- cycle;
\fill[gray!15] [-, thin] (-0.5,0.5) -- (-0.5,1) -- (1,1) -- (1,-0.25) -- (0.25,-0.25) -- cycle;
\fill[gray!55] [-, thin] (0.25,-0.25) -- (1,-0.25) -- (1,-1) -- (0.25,-1) -- cycle;
\fill[gray!15] [-, thin] (0.25,-0.25) -- (0.25,-1) -- (-1,-1) -- (-1,0.5) -- (-0.5,0.5)-- cycle;
\draw (0,0) node[above right]{{\tiny $(0,0)$}} node{{\tiny $\bullet$}};
\draw (-0.5,0.5) node[ above right]{{\tiny $(\frac{-1}{2},\frac{1}{2})$}} node{{\tiny $\bullet$}};
\draw (0.25,-0.25) node[below left]{{\tiny $(\frac{1}{4},\frac{-1}{4})$}} node{{\tiny $\bullet$}};
\draw (0.8,0.8)  node{};
\draw (0.8,-0.8)  node{};
\draw (-0.8,-0.8)  node{};
\draw (-0.8,0.8)  node{};
\draw [-, thick] (-0.5,0.5) -- (0.25,-0.25);
\draw [-, thick] (-0.5,0.5) -- (-0.5,1);
\draw [-, thick] (-0.5,0.5) -- (-1,0.5);
\draw [-, thick] (0.25,-0.25) -- (1,-0.25);
\draw [-, thick] (0.25,-0.25) -- (0.25,-1);
\end{tikzpicture}&  &
\begin{tikzpicture} 
\fill[gray!55] [-, thin] (0,0.08) -- (0,1) -- (-1,1) -- (-1,0.08) -- cycle;
\fill [gray!15] [-, thin] (0,0.08) -- (0,1) -- (1,1) -- (1,0.08)  -- cycle;
\fill[gray!55] [-, thin] (0,0.08) -- (1,0.08) -- (1,-1) -- (0,-1) -- cycle;
\fill[gray!15] [-, thin] (0,0.08) -- (0,-1) -- (-1,-1) -- (-1,0.08) -- cycle;
\draw (0,0.08) node[above right]{{\tiny $(0,\frac{1}{9})$}} node{{\tiny $\bullet$}};
\draw (0.8,0.8)  node{};
\draw (0.8,-0.8)  node{};
\draw (-0.8,-0.8)  node{};
\draw (-0.8,0.8)  node{};
\draw [-, thick] (0,0.08) -- (0,1);
\draw [-, thick] (0,0.08) -- (-1,0.08);
\draw [-, thick] (0,0.08) -- (0,-1);
\draw [-, thick] (0,0.08) -- (1,0.08);
\end{tikzpicture} &  &
\begin{tikzpicture} 
\fill[gray!55] [-, thin] (0,0.08) -- (0,1) -- (-1,1) -- (-1,0.08) -- cycle;
\fill [gray!15] [-, thin] (0,0.08) -- (0,1) -- (1,1) -- (1,0.08)  -- cycle;
\fill[gray!55] [-, thin] (0,0.08) -- (1,0.08) -- (1,-1) -- (0,-1) -- cycle;
\fill[gray!15] [-, thin] (0,0.08) -- (0,-1) -- (-1,-1) -- (-1,0.08) -- cycle;
\draw (0,0.08) node[above right]{{\tiny $(0,\frac{1}{9})$}} node{{\tiny $\bullet$}};
\draw (0.8,0.8)  node{};
\draw (0.8,-0.8)  node{};
\draw (-0.8,-0.8)  node{};
\draw (-0.8,0.8)  node{};
\draw [-, thick] (0,0.08) -- (0,1);
\draw [-, thick] (0,0.08) -- (-1,0.08);
\draw [-, thick] (0,0.08) -- (0,-1);
\draw [-, thick] (0,0.08) -- (1,0.08);
\end{tikzpicture} \\
{\small tail fan} & & {\small slice over $[0:1]$} & & slice over $[1:1]$ & & slice over $[2:3]$
\end{tabular} 
\end{center}

\ \\

The figures above represent the colored divisorial fan of a complete rational horospherical variety $X$ of complexity one with general orbit $G/H$. We only mention in the figures the non-trivial slices and the tails of the colored polyhedral divisors. The dark gray boxes correspond to polyhedral divisors defined over $\P^{1}$. The two colors of $X$ map to the vectors $e_1$, $e_2$ of the canonical basis via the map $\varrho: \F_X \to M$ defined by \eqref{varrho}. The mark in the diagram of tail fan is the color that we take into account.

Applying Theorem \ref{cox} and \cite[Ex. 3.2.3.10]{ADHL15} we obtain that the Cox ring of the variety $X$ is
$$R(X) \cong \frac{k\left[s_{1}, s_{2}, s_{3},t_1,t_2,t_3,t_4, x_1,x_2,x_3,z_1,z_2,z_3\right]}{\left(t_4^9-2t_3^9-t_1^2t_2^4,\, x_1z_1+x_2z_2+x_3z_3\right)}.$$
Moreover, from \cite[Cor. 2.12]{LT16} we determine the class group of $X$:
\begin{align*}
\Cl(X) \cong \; &\Z^{10} / \langle   f_{10}-2f_4-4f_5, \; f_{10}-9f_6, \; f_{10}-9f_7, f_8-f_4+f_5-f_3,\\ 
& f_9+f_4-f_5+f_6+f_7+f_1-f_2 \rangle \\
 \cong \; &\Z^5 \times \Z/(2) \times \Z/(9),
\end{align*} 
where we denote by $f_l$ the $l$-th vector of the canonical basis of $\Z^{10}$. 
The $\Cl(X)$-degrees of the variables can be chosen as follows.
\begin{center}
\begin{tabular}{|c|c|c|c|c|c|c|c|c|c|}
\hline
variable & $s_1$ & $s_2$ & $s_3$ & $t_1$ & $t_2$ & $t_3$ & $t_4$& $x_i$ & $z_j$ \\ \hline
degree & $f_1$ & $f_2$ & $f_3$ & $f_4$ & $f_5$ & $f_6$ & $f_7$& $f_8$ & $f_9$ \\
\hline
\end{tabular}
\end{center}
\ \\
\end{example}

\section{Proof of the main result}
We keep the same notation as in the preceding sections. We start by looking for  a natural covering of $X$ by affine open subsets. To do this, we first consider the $B$-stable dense open subset
$$X_0(\D,\F):=X(\D,\F) \setminus  \left(\bigcup_{D \in \F_{X} \setminus \F} D \right),$$
where $(\D,\F)\in\E$.
This subset is affine and intersects every $G$-orbit of $X(\D,\F)$ (compare with \cite[Lem. 2.1]{LT16}). 
Then we consider the \emph{Weyl group} $W = N_{G}(T)/T$
corresponding to a maximal torus $T$ of $B$. Note that 
for any lift $\tilde{w} \in N_{G}(T)$ of an element $w\in W$ and any $B$-stable subset $Z\subseteq X$, the subset $\tilde{w}\cdot Z \subseteq X$ does not depend on the choice of $\tilde{w}$ but only on $w$. 
Hence we may write
$$X_{0}^{w}(\D, \F):=\tilde{w} \cdot X_{0}(\D, \F).$$
The proof of the following lemma is inspired from \cite[Prop. 3.7]{AKP15}.
\begin{lemma} \label{cover}
The affine dense open subsets $X_{0}^{w}(\D,\F)$ cover $X$ when $(\D,\F)$ and $w$ run through $\E$ and $W$ respectively. 
\end{lemma} 
\begin{proof}
We may assume that $\E=\{(\D,\F)\}$. 
Let $X':=\bigcup_{w \in W} X_0^w(\D,\F)$, and let
$O=G.x$ be any $G$-orbit in $X$. By \cite[Rem. 7.2]{Tim11}, the homogeneous space $O$ is horospherical. Therefore $O=\bigcup_{w \in W} w \cdot O_{0}$, where $O_{0}$ denotes the open $B$-orbit of $O$; indeed, this result is classical for flag varieties (see e.g. \cite[Prop. 1.2.1]{Bri05}), and it extends easily to horospherical homogeneous spaces. Since $O \cap X_0$ is a $B$-stable dense open subset, we have $O_{0}\subseteq O \cap X_0$. It follows that $O \cap X'$ contains all the $w \cdot O_{0}$, where $w  \in W$, and thus $O \subseteq X'$. We conclude that $X=X'$. 
\end{proof}
The torus $\T$ identifies
naturally with the group of $G$-equivariant automorphisms of $G/H$. The trivial action of $\T$ on $\P^1$ induces an action of $\T$ on the field of rational functions $k(\P^1 \times G/H)$. We will see that this action, in turn, induces a $\T$-action on $X$. For basic notions on invariant valuations we refer to \cite[Ch. 4]{Tim11}.
\begin{lemma}  \label{Taction}
The $\T$-action on $k(\P^1 \times G/H)$ induces a regular $\T$-action on $X$ such that each affine dense open subset $X_0^w(\D,\F)$ is $\T$-stable. 
\end{lemma}

\begin{proof}
Without loss of generality, we may assume that $\E=\{(\D,\F)\}$.
Since $\T$ is connected and its action on $G/H$ commutes with the $G$-action, it preserves the open subset $w\cdot O_{0}$ and the irreducible components of its complement in $G/H$. Let $\nu$ be any $G$-valuation of $k(X)$, and let $f \in k(X)$ be a $B$-eigenfunction. Since all $B$-invariant functions in $k(G/H)$ are constant, there exists a character $\chi:\T \to k^\times$ such that $t \cdot f=\chi(t) f$ for all $t \in \T$, and thus $\nu(t \cdot f)=\nu(f)$. A $G$-valuation of $k(X)$ is determined by its values on the $B$-eigenfunctions \cite[Cor. 19.13]{Tim11}, hence we see that every $G$-valuation of $k(X)$ is also a $\T$-valuation of $k(X)$. 
Let $\V$ be the set of all $G$-valuations of $k(X)$ corresponding to proper $G$-stable closed subvarieties of $X$ which are maximal for the inclusion. By \cite[\S 13]{Tim11}, the coordinate ring of $X_0^w:=X_0^w(\D,\F)$ can be expressed inside $k(\P^1 \times G/H)$ as
\begin{equation} \label{combi chart}
 k[X_0^w]= (k(\P^1) \otimes_{k} k[w \cdot O_{0}]) \cap \left( \bigcap_{D \in \F} \O_{\nu_{w \cdot D}} \right) \cap \left( \bigcap_{\nu \in \V} \O_\nu \right),
 \end{equation}
where $\O_\nu=\{ f \in k(X)^\times \ |\  \nu(f) \geq 0\} \cup \{0\}$.
For all $w \in W$, it follows from the discussion above that $k(\P^1) \otimes_{k} k[w \cdot O_{0}]$, the local rings $\O_{\nu_{w \cdot D}}$ with $D \in \F$, and $\O_\nu$ with $\nu \in \V$ are all $\T$-stable. We deduce that the natural $G$-equivariant birational $\T$-action on $X$ is biregular on $X_0^w\subseteq X$. Since the locus where this action is biregular is a $G$-stable dense open subset of $X$, the birational $\T$-action on $X$ is biregular everywhere.
\end{proof}

Our next goal is to describe the $\T$-action on $X$ via the language of divisorial fans. Colors of $X$ are naturally represented as elements of the dual lattice $M^\vee$ as follows: For the natural action of $B$ on $k(X)$, the lattice $M$ identifies with the lattice of $B$-weights of the $B$-algebra $k(X)$. For every nonzero $B$-eigenvector $f\in k(X)$ of weight $m\in M$ and every color $D \in \F_X$, we put   
\begin{equation}  \label{varrho}
\langle m,\varrho(D)\rangle  = v_D(f),
\end{equation}
where $v_D$ is the valuation associated with $D$. Since the value $\varrho(D)$ does not depend on the choice of $f$, we obtain a map $\varrho: \F_X \to M^\vee$. We recall that the set $\F_X$ of colors of $X$ is in one-to-one correspondence with the set $\F_{G/P}$ of Schubert divisors of $G/P$. 
For all $s \in \P^1$ and $D \in \F_{G/P}$, we let $Z_s:=\{s\}\times G/P$ and $Z_D:=\P^1 \times D$. 
For all $w \in W$ and $(\D,\F) \in \E$, we define a new polyhedral divisor on $\P^1 \times G/P$ by 
$$ \QQ(w, \D,\F):= \sum_{D \in \F_{X}}\left(\varrho(D) + \sigma\right)\cdot [Z_D] + \sum_{s \in \P^1}\D_s \cdot [Z_s]+ \sum_{D \in \F_{X} \setminus \F} \emptyset \cdot [w \cdot Z_D],$$
where we denote the empty coefficients of $\QQ(w, \D,\F)$ by $\emptyset$. The tail of $\QQ(w, \D,\F)$ does not depend on $w$ and coincides with the tail $\sigma=\omega^\vee$ of $\D$.
We denote by  $\SS(\E)$ the finite set formed by the polyhedral divisors $\QQ(w, \D,\F)$ when $(\D,\F) \in \E$ and $w \in W$. Here we refer to \cite{Kno91} for the Luna--Vust theory of spherical embeddings.

\begin{proposition} \label{Tvariety}
The set $\SS(\E)$ generates a divisorial fan of $X$ as a $\T$-variety and $\QQ(w, \D,\F)$ is a polyhedral divisor of the local chart 
$X_{0}^{w}(\D,\F)$ of $X$. 
\end{proposition}
\begin{proof}
Let us show that $\QQ(w, \D,\F)$ describes the $\T$-variety $X_{0}^{w}(\D,\F)$.
Without loss of generality, we may assume that $w$ is the neutral element.
Let $O_{0}$ be the open $B$-orbit of $G/H$. Note that $k[O_{0}]$
is naturally equipped with 
an $M$-graduation arising from the $\T$-action on $G/H$. Given any toroidal $G$-equivariant embedding $X'$ 
of $G/H$, the prime $G$-divisors of $X'$ are exactly the $\T$-divisors
that are in the complement of $G/H$. Comparing the divisorial fans
for the $\T$-actions on $G/H$ and $X'$ given in \cite[Th. 1.1]{AKP15} and using \cite[Prop. 3.13]{PS11}, we deduce that for any $G$-valuation $v$ of $k(G/H)$ restricted on $k[O_{0}]$, there exists a unique $\gamma\in M^{\vee}_{\Q} = \Q\otimes_{\Z}M^{\vee}$ such that
$v(f) = \langle \gamma, m\rangle $
for any nonzero element $f$ of degree $m\in M$, and vice-versa. In the following we will denote by $v_{\gamma}$ the $G$-valuation attached to $\gamma$. 

Let $O'$ be  the complement in $G/H$ of the union of the colors of $\F_X \setminus \F$.
With the same notation as in the proof of Lemma \ref{Taction}, Equation \eqref{combi chart} yields
\begin{equation} \label{eq3}
k[X_0(\D,\F)]=(k(\P^1) \otimes_{k} k[O']) \cap \left(\bigcap_{v \in \V}\O_v \right).
\end{equation} 
For any $s\in\P^{1}$, we denote by $C(\D_{s})$ the cone in $\Q\oplus M^{\vee}_{\Q}$ generated by the union of $\{0\}\times \sigma$ and $\{1\}\times \D_{s}$.  Considering $f_{1}\in k(\P^{1})$ and $f_{2}\in k[O']$ of degree $m\in M$, we have the following equivalences: $$f_{1}\otimes f_{2} \in k[X_0(\D,\F)] \setminus \{0\} \Leftrightarrow \nu(f_{1}\otimes f_{2})\geq 0 \text{ for all } \nu\in \V$$
$$\Leftrightarrow \beta\cdot \ord_{s}(f_{1}) + v_{\gamma}(f_{2})\geq 0 \text{ for all } s\in \P^{1} \text{ and } (\beta, \gamma)\in C(\D_{s})$$
$$\Leftrightarrow m \in   \sigma^\vee\cap M \text{ and }\div(f_{1}\otimes f_{2}) + \QQ(w, \D,\F)(m)\geq 0.$$
The first equivalence is a rephrasing of Equality \eqref{eq3}.
The second equivalence follows from the descriptions of $\V$
in terms of $\D$; compare with \cite[\S\S 1.2.3 and 1.3.3]{LT16}. 
The condition $f_{1}\otimes f_{2}\in k(\P^1) \otimes_k k[O']$
gives 
$$\ord_{Z_{D}}(f_{1}\otimes f_{2}) + \langle \varrho(D), m \rangle \geq 0$$
for all $D\in \F$ (which holds by the description of $O'$ in terms 
of divisorial fans). Moreover the conditions $$\beta\cdot \ord_{s}(f_{1}) + v_{\gamma}(f_{2}) = \beta\cdot \ord_{s}(f_{1}) + \langle \gamma, m\rangle \geq 0,$$ 
for all $(\beta, \gamma)\in C(\D_{s})$ and $s\in \P^{1}$, are translated into the conditions $m\in\sigma^{\vee}\cap M$ and
$$\ord_{Z_{s}}(f_{1}\otimes f_{2}) + \min_{\gamma\in \D_{s}}\langle \gamma, m\rangle\geq 0,$$ yielding the last equivalence. 

The verification that $\SS(\E)$ generates a divisorial fan consists in describing the 
polyhedral divisors associated with the intersections of the $X_0^{w}(\D,\F)$'s
and is left to the reader.
\end{proof}

\begin{proof}[Proof of Theorem \ref{cox}] 
Starting with the set $\SS(\E)$ generating a divisorial fan of the $\T$-variety $X$, we investigate the 
Cox ring of $X$ by using \cite[Th. 4.8]{HS10}. To do this, we apply \cite[Prop. 3.13]{PS11} to determine the prime $\T$-divisors on $X$ that are described by some
extremal rays of tails of elements of $\SS(\E)$ or by some vertices of non-trivial slices. From the study of $G$-invariant valuations in the proof of Proposition
\ref{Tvariety}, we know that the prime $\T$-divisors given by the extremal rays coincide with the prime $G$-divisors parameterized by the set $\Ray(\E)$. 
The remaining prime $\T$-divisors are given either by the vertices of slices of $\SS(\E)$ over the $Z_{y_{i}}$'s, and then they coincide with the prime $G$-divisors given by the pairs $(y_{i}, v)$, where $v\in \E_{y_{i}}$ and $1\leq i\leq r$, or else they are given by the vertices of slices of $\SS(\E)$ over the $Z_{D}$'s.  
The latter will not play any role in the description of $R(X)$ since their vertices $\varrho(D)$ are lattice vectors.
Therefore, by \cite[Th. 4.8]{HS10}, we obtain that $R(X)$ is isomorphic to
$$k \left[S_{\rho} ;\, \rho\in\Ray(\E) \right] \otimes_{k}  R(\P^{1}\times G/P)\left[T_{(y_{i}, v)} \,;\, (y_{i}, v)\in \Vert(\E)\right]/I,$$
where $I$ is the ideal generated by 
$$-s_{i} + \prod_{v \text{ vertex of } \E_{y_i}}  T_{(y_{i},v)}^{\mu(v)}$$ 
for $1\leq i\leq r$, with $s_i$ the canonical section of $Z_{y_{i}}$ and $\mu(v)$ as in the statement of Theorem \ref{cox}. 
We conclude by using
that $R(\P^{1}\times G/P)$ identifies with $k[T_{0}, T_{1}]\otimes_{k}R(G/P)$; compare with \cite[Lem. 4.2.2.3]{ADHL15}. Also, via this isomorphism, one can take 
$s_i = \alpha_i T_0+\beta_i T_1$, where $y_{i} = [\alpha_i, \beta_i]$. Finally, the $\Cl(X)$-degree of the variables $S_\rho$ and $T_{(y_{i}, v)}$ is given by \cite[Th. 4.8]{HS10}, and we determine the $\Cl(X)$-grading on $R(G/P)$ by using \cite[Prop. 4.7]{HS10}.
\end{proof}

\noindent {\em Acknowledgments.} 
We would like to thank the anonymous referee and Michel Brion for their valuable comments which allowed to improve the presentation.
The first-named author is supported by the University Heinrich Heine of D\"{u}sseldorf. 
The second-named author is grateful to the Max-Planck-Institut f\"{u}r Mathematik of Bonn for the warm hospitality and support provided during the writing of this paper.

\bibliographystyle{alpha}

\end{document}